\documentclass{amsart}

\usepackage{amsthm,amssymb,amsmath,a4wide,graphicx,tikz}
\usepackage[english]{babel}
\usepackage{subfigure}
\usepackage{eepic}

\title{Induced Random $\beta$-transformation}

\begin{document}
\newtheorem{Theorem}{Theorem}[section]
\newtheorem{Proposition}{Proposition}[section]
\newtheorem{Lemma}{Lemma}[section]
\newtheorem{Corollary}{Corollary}[section]
\newtheorem{Conjecture}{Conjecture}[section]
\newtheorem{Definition}{Definition}[section]
\newtheorem{Exercise}{Exercise}[section]
\newtheorem{Example}{Example}[section]
\newtheorem{Examples}{Examples}[section]
\newtheorem{Remark}{Remark}[section]
\newtheorem{Remarks}{Remarks}[section]

\newcommand{\R}{\mathbb R}
\newcommand{\Q}{\mathbb Q}
\newcommand{\Z}{\mathbb Z}
\newcommand{\N}{\mathbb{N}}
\newcommand{\lbetar}{\lfloor \beta \rfloor}

\author{Simon Baker
\and
Karma Dajani}

\address{
Department of Mathematics and Statistics, University of Reading, Reading,  RG6 6AX, UK. 
E-mail: simonbaker412@gmail.com}
\address{
Department of Mathematics, Utrecht University, 3508 TA Utrecht, The Netherlands.
E-mail: k.dajani1@uu.nl}
\date{}
\begin{abstract}
In this article we study the first return map defined on the switch region induced by the greedy and lazy maps. In particular we study the allowable sequences of return times, and when the first return map is a generalised L\"uroth series transformation. We show that there exists a countable collection of disjoint intervals $(\mathcal{I}_{n})_{n=1}^{\infty},$ such that all sequences of return times are permissible if and only if $\beta\in \mathcal{I}_{n}$ for some $n$. Moreover, we show that there exists a set $M\subseteq(1,2)$ of Hausdorff dimension $1$ and Lebesgue measure zero, for which the first return map is a generalised L\"uroth series transformation if and only if $\beta\in M$.

\end{abstract}
\maketitle

\section{Introduction}
Let $\beta\in(1,2)$ and $I_{\beta}:=[0,\frac{1}{\beta-1}]$. Given $x\in I_{\beta}$ we call a sequence $(b_{n})_{n=1}^{\infty}\in\{0,1\}^{\mathbb{N}}$ a $\beta$-expansion for $x$ if $$x=\sum_{n=1}^{\infty}\frac{b_{n}}{\beta^{n}}.$$ Non-integer representations of real numbers were pioneered in the papers of R\'enyi \cite{Renyi} and Parry \cite{Parry}. Since then they have been studied by many authors and have connections with ergodic theory, fractal geometry, and number theory (see the survey articles \cite{Kom} and \cite{Sid2}). Perhaps one of the most interesting objects to study within expansions in non-integer bases is the set of expansions, i.e., $$\Sigma_{\beta}(x):=\Big\{(b_{n})_{n=1}^{\infty}\in\{0,1\}^{\mathbb{N}}:\sum_{n=1}^{\infty}\frac{b_{n}}{\beta^{n}}=x\Big\}.$$ A result of Sidorov states that given $\beta\in(1,2)$ then Lebesgue almost every $x\in I_{\beta}$ satisfies $\textrm{card }\Sigma_{\beta}(x)=2^{\aleph_{0}}$ \cite{Sid}. Moreover, for any $k\in\mathbb{N}\cup\{\aleph_{0}\}$ there exists $\beta\in(1,2)$ and $x\in I_{\beta}$ such that $\textrm{card } \Sigma_{\beta}(x)=k,$ see \cite{DaKa,EHJ,EJ}. The situation described above is completely different to the case of integer base expansions where every number has a unique expansion except for a countable set of exceptions which have precisely two.

A useful observation when studying expansions in non-integer bases is that a $\beta$-expansion has a natural dynamical interpretation. Namely, let $T_0(x)=\beta x,$ $T_1(x)=\beta x-1$,
 and $$\Gamma_{\beta}(x):=\Big\{(a_{n})_{n=1}^{\infty}\in \{T_{0},T_{1}\}^{\mathbb{N}}:(a_{n}\circ \cdots \circ a_{1})(x)\in I_{\beta}
 \textrm{ for all } n\in\mathbb{N}\Big\}.$$ It was shown in \cite{Baker} that $\textrm{card }\Sigma_{\beta}(x)=\textrm{card }\Gamma_{\beta}(x)$ and the map sending $(b_{n})$ to $(T_{b_{n}})$ is a bijection between these two sets. As such, performing the map $T_{0}$ corresponds to taking the digit $0,$ and $T_{1}$ corresponds to taking the digit $1$. An all encompassing method by which we can use the maps $T_{0}$ and $T_{1}$ to generate $\beta$-expansions is the {\it random $\beta$-transformation}. This map is defined as follows. Set $\Omega=\{0,1\}^{\mathbb{N}}$ and denote by $\sigma$ the left shift on $\Omega$.
Consider the transformation $K_{\beta}:\Omega \times [0, \displaystyle\frac{1}{\beta -1}]\to \Omega \times [0, \displaystyle\frac{1}{\beta -1}]$ defined by
\[ K_{\beta} (\omega, x) = \left\{
\begin{array}{ll}
(\omega, T_0 x), & \text{if } 0\le x<\frac{1}{\beta},\\
\\
(\sigma \omega, T_{\omega_1} x), & \text{if } \frac{1}{\beta}\le x \le \frac{1}{\beta (\beta-1)},\\
\\
(\omega, T_1 x), & \text{if } \frac{1}{\beta (\beta-1)}< x\le \frac{1}{\beta -1}.
\end{array}
\right.\]

The random $\beta$-transformation $K_{\beta}$ was introduced and studied in \cite{DK,DdV1,DdV2}. Given $x\in I_{\beta}$, the map $K_{\beta}$ generates all possible $\beta$-expansions of $x$. Furthermore, it is a random mix of the classical greedy and lazy maps defined by

\[ G_{\beta} (x) = \left\{
\begin{array}{ll}
T_0 (x), & \text{if } 0\le x<\frac{1}{\beta},\\
\\
T_1 (x), & \text{if } \frac{1}{\beta} \le x\le \frac{1}{\beta -1},
\end{array}
\right.\]
and
\[ L_{\beta} (x) = \left\{
\begin{array}{ll}
T_0 (x), & \text{if } 0\le x<\frac{1}{\beta (\beta-1)},\\
\\
T_1 (x), & \text{if } \frac{1}{\beta (\beta-1)} \le x\le \frac{1}{\beta -1}
\end{array}
\right.\]
respectively. Let $S:=[\displaystyle\frac{1}{\beta}, \displaystyle\frac{1}{\beta (\beta-1)}]$, we refer to $S$ as the {\it switch region}.  This is the region where the greedy map $G_{\beta}$ and lazy map $L_{\beta}$ differ, and is the region where the coordinates of $\omega$ are used to decide which map to use. Understanding the dynamics of the maps $T_{0}$ and $T_{1}$ on the switch region provides valuable insight into the possible $\Gamma_{\beta}(x),$ and thus the possible $\Sigma_{\beta}(x)$.

This paper is concerned with the dynamics of the first return map defined on the switch region. We consider the induced transformation $U_{\beta}$ of $K_{\beta}$ on the set $\Omega \times S$. More precisely,  $U_{\beta}: \Omega \times S\to \Omega \times S$ is defined as follows: $$U_{\beta}(\omega,x):=K_{\beta}^{r_1(\omega,x)}(\omega,x), \textrm{ where }
r_1(\omega,x)=\inf\{m\ge 1: K_{\beta}^m(\omega,x)\in \Omega\times S\}.$$  Similarly we set
$$U_{\beta,0}(x):=U_{\beta}((0)^{\infty},x) \textrm{ and } U_{\beta,1}(x):=U_{\beta}((1)^{\infty},x).$$ Note that when we have fixed the sequence $\omega$ to equal $(0)^{\infty}$ or $(1)^{\infty}$ the maps $U_{\beta,0}$ and $U_{\beta,1}$ are well defined maps from $S$ to $S$.

\begin{Remark}
The map $U_{\beta}$ is defined on $\Omega\times S,$ and both $U_{\beta,0}$ and $U_{\beta,1}$  are defined on $S$. However, there exists $\omega$ and $x$ for which $K_{\beta}(\omega,x)$ is never mapped back into $\Omega\times S$, thus for this choice of $\omega$ and $x$ the map $U_{\beta}$ is not well defined. Similarly, there exists $x$ for which $U_{\beta,\omega_{i}}$ is not well defined. However, it is a consequence of the work of Sidorov \cite{Sid} that the set of $x$ for which $U_{\beta}^{n}(w,x)$ is well defined for all $n\in\mathbb{N}$ and $\omega\in \Omega$ is of full Lebesgue within $S$. Similarly, the set of $x$ for which $U_{\beta,\omega_{i}}^{n}(x)$ is well defined for every $n\in\mathbb{N}$ is of full Lebesgue measure within $S$. Throughout this article we will abuse notation and let $S$ denote both the switch region and the full measure subset of $S$ for which $U_{\beta}$ and $U_{\beta,\omega_{i}}$ are well defined. It should be clear which interpretation of $S$ we mean from the context.
\end{Remark}

For $i\ge 1$ let
$r_i(\omega,x):=r_1(U_{\beta}^{i-1}(\omega,x))$ be the $i$th return time to the switch region $\Omega\times S$. Note that for any $\beta$ and $\omega,$ the set $\{r_{1}(\omega,x)\}_{x\in S}$ equals $\mathcal{R}_{\beta}:=\{m,m+1,\ldots\}$ where $m$ is some natural number that only depends upon $\beta$. We emphasise that $\mathcal{R}_{\beta}$ has no dependence on $\omega$. One of the goals of this paper is to understand the sequences $(r_i(\omega,x))_{i=1}^{\infty}$ and to answer the following question: given $\omega\in \Omega$ and a sequence of integers $(j_{i})_{i=1}^{\infty}\in \mathcal{R}_{\beta}^{\mathbb{N}}$, when is it possible to find $x\in S$ such that $r_i(\omega,x)=j_i$, for $i=1,2,\ldots$? The following theorem provides an answer to this question. Before we state this theorem we have to introduce two classes of algebraic integers. Let $\alpha_{k}$ denote the unique solution in $(1,2)$ of the equation $$x^{k+1}-2x^{k}+x-1=0,$$ and let $\gamma_{k}$ denote the $k$-th multinacci number. Recall that the $k$-th multinacci number is the unique root of $$x^{k+1}-x^{k}-x^{k-1}-\cdots -x-1=0$$ contained in $(1,2)$.

\begin{Theorem}
\label{Free return times}
Let $\beta\in (\alpha_{k},\gamma_{k}]$ for some $k\geq 2,$ then for any $\omega\in\Omega$ and $(j_{i})\in \mathcal{R}_{\beta}^{\mathbb{N}}$ there exists $x\in S$ such that $r_i(\omega,x)=j_i$, for $i=1,2,\ldots$. Moreover, if $\beta\notin (\alpha_{k},\gamma_{k}]$ for all $k\geq 2,$ then there exists $\omega\in\Omega$ and $(j_{i})\in \mathcal{R}_{\beta}^{\mathbb{N}}$ such that no $x\in S$ satisfies $r_i(\omega,x)=j_i$, for $i=1,2,\ldots$.
\end{Theorem}

As we will see, the algebraic properties of $\alpha_{k}$ and $\gamma_{k}$ correspond naturally to conditions on the orbit of $1$ and its reflection $\frac{1}{\beta-1} -1$. These points determine completely the dynamics of the greedy map $G_{\beta}$ and lazy map $L_{\beta}$ respectively, and hence it is not surprising that these points play a crucial role in our situation as well. For values of $\beta$ lying outside of the intervals $(\alpha_{k},\gamma_{k}]$ it is natural to ask whether the following weaker condition is satisfied: given $(j_{i})\in\mathcal{R}_{\beta}^{\mathbb{N}}$ does there exist $\omega\in\Omega$ and $x\in S$ such that $r_i(\omega,x)=j_i$ for $i=1,2,\ldots$. Let $\eta_{k}$ denote the unique root of the equation $$2x^{k+1}-4x^{k}+1=0$$ contained in $(1,2).$

\begin{Theorem}
\label{Restricted omega}
Let $\beta\in (\alpha_{k},\eta_{k}]$ for some $k\geq 1$, then for any sequence $(j_{i})\in\mathcal{R}_{\beta}^{\mathbb{N}}$ there exists $\omega\in\Omega$ and $x\in S$ such that $r_i(\omega,x)=j_i$ for $i=1,2,\ldots$.
\end{Theorem} If $\beta$ satisfies the hypothesis of Theorem \ref{Restricted omega} then the orbit of $1$ and $\frac{1}{\beta-1}-1$ satisfy a \emph{cross over property.} This cross over property is sufficient to prove Theorem \ref{Restricted omega}. Note that $\alpha_{k}\leq \gamma_{k}\leq \eta_{k}$ for each $k\geq 1$. We include a tables of values for $\alpha_{k}$, $\gamma_{k}$ and $\eta_{k}$ in Figure \ref{fig2}.
\begin{figure}[t]
\centering \unitlength=0.70mm
\begin{tabular}{| l | l | l | l |}
    \hline
    $k$ & $\alpha_{k}$ & $\gamma_{k}$ & $\eta_{k}$ \\ \hline
    $1$ & $\frac{1+\sqrt{5}}{2}$ & $\frac{1+\sqrt{5}}{2}$ & $1+2^{-1/2}$  \\ \hline
    $2$ & $1.7549\ldots$ & $1.8393\ldots$ & $1.8546\ldots$ \\ \hline
    $3$ & $1.8668\ldots$ & $1.9276\ldots$ & $1.9305\ldots$ \\ \hline
    $4$ & $1.9332\ldots$ & $1.9660\ldots$ & $1.9666\ldots$ \\ \hline
    $5$ & $1.9672\ldots$ & $1.9836\ldots$ & $1.9837\ldots$ \\ \hline
    \end{tabular}
\caption{Tables of values for $\alpha_{k}$, $\gamma_{k}$ and $\eta_{k}$}
    \label{fig2}
\end{figure}

The second half of this paper is concerned with the maps $U_{\beta,0}$ and $U_{\beta,1}.$ Before we state our results it is necessary to make a definition. Given a closed interval $[a,b]$, we call a map $T: [a,b]\to [a,b]$ a generalized L\"uroth series transformation (abbreviated to GLST) if there exists a countable set of bounded subintervals $\{I_{n}\}_{n=1}^{\infty}$ ($I_{n}=(l_n,r_n),[l_n,r_n],(l_n,r_n],[l_n,r_n)$) for which the following criteria are satisfied:
\begin{enumerate}
\item $I_{n}\cap I_{m}=\emptyset$ for $n\neq m$.
  \item $\sum_{n=1}^{\infty}(r_{n}-l_{n})=b-a$.
  \item $$T(x)= a +\frac{(x-l_{n})(b-a)}{r_{n}-l_{n}}$$ for $x\in I_{n}$.
\end{enumerate}Property $(3)$ is equivalent to the map $T$ restricted to the interval $I_{n}$ being the unique surjective linear orientation preserving map from the interval $I_{n}$ into $S$.

The traditional L\"uroth expansion of a number $x\in(0,1]$ is a sequence of natural numbers $(a_{n})_{n=1}^{\infty}$ where each $a_{n}\geq 2$ and $$x= \frac{1}{a_{1}}+\frac{1}{a_{1}(a_{1}-1)a_{2}}+\cdots+\frac{1}{a_{1}(a_{1}-1)a_{2}(a_{2}-1)\cdots a_{n}}+\cdots.$$ This L\"uroth expansion $(a_{n})$ can be seen to be generated by the map $T:[0,1]\to[0,1]$ where
$$T(x)=
\begin{cases}
    n(n+1)x-n,& \text{if } x\in(\frac{1}{n+1},\frac{1}{n}]\\
    0,              & \text{if } x=0
\end{cases} $$
GLST's were introduced in \cite{BBDK}. Our definition is slightly different to that appearing in this paper but all of the main results translate over into our context. Namely if $T:[a,b]\to[a,b]$ is a GLST then the normalised Lebesgue measure on $[a,b]$ is a $T$-invariant ergodic measure. Our main result for the maps $U_{\beta,0}$ and $U_{\beta,1}$ is the following theorem.

\begin{Theorem}
\label{Luroth return map}
There exists a set $M\subseteq(1,2)$ of Hausdorff dimension 1 and Lebesgue measure zero such that:
\begin{enumerate}
  \item If $\beta\in M$ then both $U_{\beta,0}$ and $U_{\beta,1}$ are GLSTs.
  \item If $\beta\notin M$ then both $U_{\beta,0}$ and $U_{\beta,1}$ are not GLSTs.
\end{enumerate}
\end{Theorem}What is more we can describe the set $M$ explicitly.

\bigskip
\noindent

Before we move on to our proofs of Theorems \ref{Free return times}, \ref{Restricted omega} and \ref{Luroth return map} we provide a worked example. Namely we consider the case where $\beta = \frac{1+\sqrt{5}}{2}.$ This case exhibits some of the important features of our later proofs.

\begin{Example}
\label{Golden ratio example}
When $\beta = \frac{1+\sqrt{5}}{2}$ then $S=[\frac{1}{\beta}, 1]$. Let $C_j=\{\omega\in \Omega: \omega_1=j\}$, $j=0,1$, then for any $\omega\in C_0$, $r_1(\omega, 1)=\infty$, and $r_1(\omega, \frac{1}{\beta})=1$, while for any $\omega\in C_1$, we have $r_1(\omega, 1)=1$ and $r_1(\omega, \frac{1}{\beta})=\infty$. If  $x\in (\frac{1}{\beta}, 1)$, then $r_1(\omega ,x)\ge 2$ for all $\omega \in \Omega$.

\medskip
Let
\begin{equation}
\label{B0*}
B_i^0:=\{x\in S:U_{\beta,0}(x)=(T_{1}^{i-1}\circ T_0)(x)\}
\end{equation}and
\begin{equation}
\label{B1*}
B_i^1:=\{x\in S:U_{\beta,1}(x)=(T_{0}^{i-1}\circ T_1)(x)\}
\end{equation}where $i\ge 2$. A simple calculation shows that
\begin{equation}
\label{B0}
B_i^0=\Big(\sum_{n=2}^{i+1}\frac{1}{\beta^{n}},\sum_{n=2}^{i+2}\frac{1}{\beta^{n}} \Big]=(T_{1}^{i-1}\circ T_0)^{-1}\Big(\frac{1}{\beta},1\Big]
\end{equation}and
\begin{equation}
\label{B1}
B_i^1=\Big[\frac{1}{\beta}+\frac{1}{\beta^{i+1}},\frac{1}{\beta}+\frac{1}{\beta^{i}})=(T_{0}^{i-1}\circ T_1)^{-1}\Big[\frac{1}{\beta},1\Big),
\end{equation}where $i\ge 2$.

The collection $\{B_i^0: i\ge 2\}$ is a partition of $(\frac{1}{\beta},1)$, and  $\{B_i^1: i\ge 2\}$ is a partition of $(\frac{1}{\beta},1)$. Equation (\ref{B0}) demonstrates that $U_{\beta,0}$ restricted to $B_i^0$ is a full branch, thus $U_{\beta,0}$ is a GLST. Similarly equation (\ref{B1}) implies $U_{\beta,1}$ is a GLSTs. We include a diagram of the graph of $U_{\beta,0}$ in Figure \ref{fig1}.

By the aforementioned results of \cite{BBDK} we know that a $GLST$ is ergodic with respect to the normalised Lebesgue measure $\mu$. As such we can state the average return time. For $\beta=\frac{1+\sqrt{5}}{2}$ Lebesgue almost every $x\in S$ satisfies

\begin{align*}
\lim_{n\to\infty}\frac{1}{n}\sum_{j=0}^{n-1}r_{j}((0)^{\infty},x)&=\lim_{n\to\infty}\frac{1}{n}\sum_{j=0}^{n-1}\sum_{i=2}^{\infty}i\chi_{B_{i}^{0}}((U_{\beta,0})^{j}(x))\\
&=\int \sum_{i=2}^{\infty}i\chi_{B_{i}^{0}}\, d\mu\\
&= 2\beta^{2}-\beta\\
&=3.6178\ldots.
\end{align*}Where in the above $\chi_{B_{i}^{0}}$ denotes the characteristic function on $B_{i}^{0}.$ Note that the result stated above holds with  $r_{j}((0)^{\infty},x)$ replaced with $r_{j}((1)^{\infty},x).$

By Theorem \ref{Free return times} we know that there exists $\omega\in \Omega$ and $(j_{i})_{i=1}^{\infty}\in R_{\frac{1+\sqrt{5}}{2}}^{\mathbb{N}}$ for which no $x$ satisfies $r_{i}(\omega,x)=j_{i}$ for $i=1,2,\ldots.$ This is essentially a consequence of the fact mentioned above that if $\omega\in C_0$ then $r_1(\omega, 1)=\infty$, and $r_1(\omega, \frac{1}{\beta})=1$, while for any $\omega\in C_1$, we have $r_1(\omega, 1)=1$ and $r_1(\omega, \frac{1}{\beta})=\infty$. This statement implies that we cannot have a return time $1$ followed by any other natural number.

\end{Example}

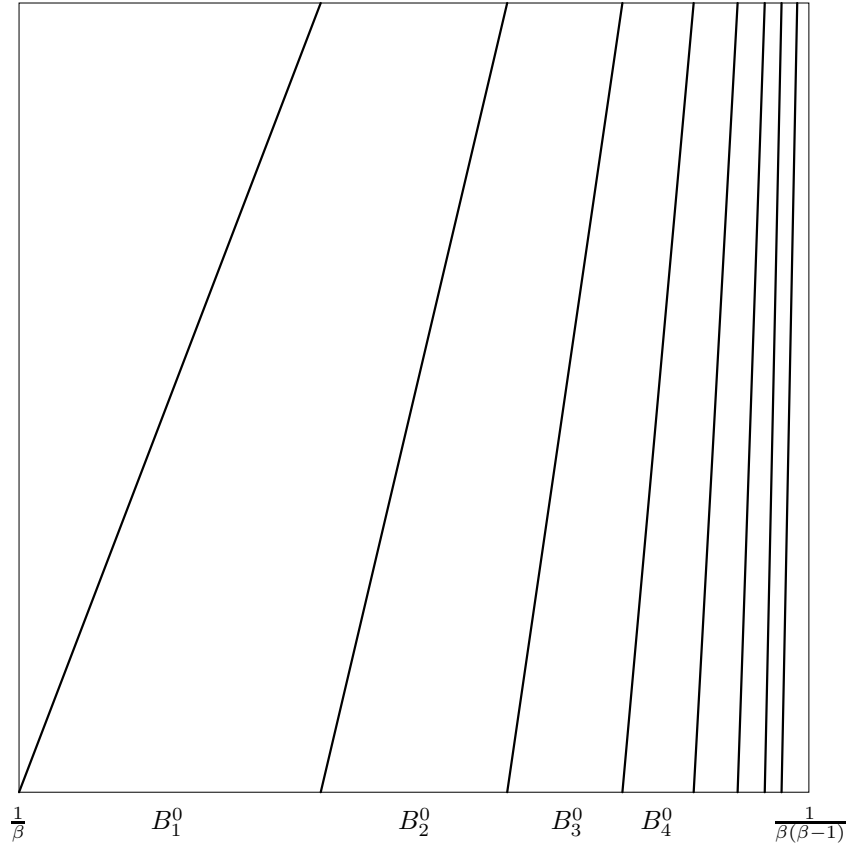
\begin{figure}[t]
\centering \unitlength=0.70mm
\begin{picture}(150,160)(0,-10)
\thinlines
\path(0,0)(0,150)(150,150)(150,0)(0,0)

\put(-2,-7){$\frac{1}{\beta}$}
\put(143,-7){$\frac{1}{\beta(\beta-1)}$}
\put(25,-7){$B_{1}^{0}$}
\put(72,-7){$B_{2}^{0}$}
\put(101,-7){$B_{3}^{0}$}
\put(118,-7){$B_{4}^{0}$}

\thicklines
\path(0,0)(57.29,150)
\path(57.29,0)(92.7,150)
\path(92.7,0)(114.58,150)
\path(114.58,0)(128.1,150)
\path(128.1,0)(136.45,150)
\path(136.45,0)(141.616,150)
\path(141.616,0)(144.80,150)
\path(144.80,0)(147.77,150)
\end{picture}
\caption{The graph of $U_{\beta,0}$ when $\beta=\frac{1+\sqrt{5}}{2}$}
    \label{fig1}
\end{figure}
\section{Sequences of return times}
\subsection{Proof of Theorem \ref{Free return times}}
In this section we prove Theorem \ref{Free return times}. The proofs of Theorems \ref{Free return times} and \ref{Restricted omega} both make use of a nested interval construction. We begin by examining the condition $\beta\in(\alpha_{k},\gamma_{k}]$. It is easy to show that the following statements hold:
\begin{equation}
\label{hop 1}
\beta\in(\alpha_{k},2) \iff (T_{1}^{k-1}\circ T_{0})\Big(\frac{1}{\beta}\Big)>\frac{1}{\beta(\beta-1)} \iff (T_{0}^{k-1}\circ T_{1})\Big(\frac{1}{\beta(\beta-1)}\Big)<\frac{1}{\beta}
\end{equation}
and
\begin{equation}
\label{hop 2}
\beta\in(1,\gamma_{k}] \iff (T_{1}^{k}\circ T_{0})\Big(\frac{1}{\beta}\Big)\leq \frac{1}{\beta}\iff (T_{0}^{k}\circ T_{1})\Big(\frac{1}{\beta(\beta-1)}\Big)\geq\frac{1}{\beta(\beta-1)}.
\end{equation}Thus $\beta\in(\alpha_{k},\gamma_{k}]$ is equivalent to the orbit of ${\frac{1}{\beta}}$ either jumping over the switch region, or satisfying $U_{\beta,0}(\frac{1}{\beta})=(T_{1}^{k}\circ T_{0})(\frac{1}{\beta})=\frac{1}{\beta}$. Similarly, $\beta\in(\alpha_{k},\gamma_{k}]$ is equivalent to the orbit of $\frac{1}{\beta(\beta-1)}$ either jumping over the switch region, or satisfying $U_{\beta,1}(\frac{1}{\beta(\beta-1)})=(T_{0}^{k}\circ T_{1})(\frac{1}{\beta(\beta-1)})=\frac{1}{\beta(\beta-1)}$. The following properties are important consequences of the above. First of all it is straightforward to see that for $\beta\in (\alpha_{k},\gamma_{k}]$ we have $R_{\beta}=\{k+1,k+2,\ldots\}$. Secondly we have
\begin{equation}
\label{Full 1}
B_{i}^{0}=(T_{1}^{i-1}\circ T_{0})^{-1}(S)
\end{equation} and
\begin{equation}
\label{Full 2}
B_{i}^{1}=(T_{0}^{i-1}\circ T_{1})^{-1}(S)
\end{equation}for $i\geq k+1$ where $B_{i}^{0}$ and $B_{i}^{1}$ are as in Example \ref{Golden ratio example}, but in this case they do not form a partition of $S$. We now prove Theorem \ref{Free return times}, we separate our proof into the following propositions..

\begin{Proposition}
\label{Prop1}
Let $\beta\in (\alpha_{k},\gamma_{k}]$ for some $k\geq 2,$ then for any $\omega\in\Omega$ and $(j_{i})\in \mathcal{R}_{\beta}^{\mathbb{N}}$ there exists $x\in S$ such that $r_i(\omega,x)=j_i$, for $i=1,2,\ldots$.
\end{Proposition}
\begin{proof}
Let $\beta\in(\alpha_{k},\gamma_{k}]$ and let us fix $(\omega_{i})\in \Omega$ and $(j_{i})\in\{k+1,k+2,\ldots,\}^{\mathbb{N}}$. We let $\mathcal{I}_{1}=B_{j_{1}}^{\omega_{1}}$ and
\begin{equation}
\label{inclusion}
\mathcal{I}_{i}:=B_{j_{1}}^{\omega_{1}}\cap (T_{\overline{\omega_{1}}}^{j_{1}-1}\circ T_{\omega_{1}})^{-1}(B_{j_{2}}^{\omega_{2}})\cap \cdots \cap \Big((T_{\overline{\omega_{i-1}}}^{j_{i-1}-1}\circ T_{\omega_{i-1}})\circ \cdots \circ (T_{\overline{\omega_{1}}}^{j_{1}-1}\circ T_{\omega_{1}})\Big)^{-1}(B_{j_{i}}^{\omega_{i}})
 \end{equation}for $i\geq 2.$ In the above and throughout we let $\overline{\omega_{i}}=1-\omega_{i}$. Any element of $\mathcal{I}_{i}$ satisfies $r_{l}(\omega,x)=j_{l}$ for $1\leq l\leq i$. Note that by Equations (\ref{Full 1}) and (\ref{Full 2}) we have $(T_{\overline{\omega_{1}}}^{j_{1}-1}\circ T_{\omega_{1}})(\mathcal{I}_{1})=S,$ by an induction argument it can be shown that
\begin{equation}
\label{Full Ei}
\Big((T_{\overline{\omega_{i}}}^{j_{i}-1}\circ T_{\omega_{i}})\circ \cdots \circ (T_{\overline{\omega_{1}}}^{j_{1}-1}\circ T_{\omega_{1}})\Big)(\mathcal{I}_{i})=S
\end{equation} for all $i\in\mathbb{N}$. Equation (\ref{Full Ei}) guarantees that $\mathcal{I}_{i}$ is nonempty and well defined for each $i\in\mathbb{N}$. Moreover, $\mathcal{I}_{i+1}\subseteq \mathcal{I}_{i}$ by equation (\ref{inclusion}). Thus $(\mathcal{I}_{i})$ is a decreasing sequence of compact intervals and $$E=\bigcap_{i=1}^{\infty}\mathcal{I}_{i}$$ is nonempty. Finally, any $x\in E$ satisfies $r_{i}(\omega,x)=j_{i}$ for all $i\in\mathbb{N}$.
\end{proof}

\begin{Proposition}
\label{Prop2}
Let $\beta\in (1,\frac{1+\sqrt{5}}{2}],$ then there exists $\omega\in\Omega$ and $(j_{i})\in \mathcal{R}_{\beta}^{\mathbb{N}}$ such that no $x\in S$ satisfies $r_i(\omega,x)=j_i$, for $i=1,2,\ldots$.
\end{Proposition}
\begin{proof}
Any $\beta\in (1,\frac{1+\sqrt{5}}{2}]$ satisfies $R_{\beta}=\{1,2,\ldots\}.$ We now fix the sequence $\omega=(0)^{\infty}$ and $(j_{i})=(1)^{\infty}.$ There exists no $x\in S$ satisfying $r_{i}((0)^{\infty},x)=1$ for all $i\geq 1,$ as this would imply there exists $x\in S$ satisfying $T_{0}^{i}(x)\in S$ for all $i\geq 1.$ This is not possible as repeated iteration of $T_{0}$ eventually maps any element of $S$ outside of $S$.
\end{proof}

\begin{Proposition}
\label{Prop3}
Let $\beta\in(\frac{1+\sqrt{5}}{2},2)\setminus \cup_{k=2}^{\infty}(\alpha_{k},\gamma_{k}],$ then there exists $\omega\in\Omega$ and $(j_{i})\in \mathcal{R}_{\beta}^{\mathbb{N}}$ such that no $x\in S$ satisfies $r_i(\omega,x)=j_i$, for $i=1,2,\ldots$.
\end{Proposition}
\begin{proof}
For $\beta\in (\frac{1+\sqrt{5}}{2},2)$ we have $r_{1}(\omega,1/\beta)\geq 2$ for any $\omega\in C_{0}$. Moreover, by our assumption that $\beta\notin (\alpha_{k},\gamma_{k}]$ for any $k\geq 2$ we must have $$(T_{1}^{k}\circ T_{0})\Big(\frac{1}{\beta}\Big)\in\Big(\frac{1}{\beta},\frac{1}{\beta(\beta-1)}\Big]$$ for some $k\geq 1$. For such a $\beta$ we have $R_{\beta}:=\{k+1,k+2,\ldots,\}$. Let $\omega=(0)^{\infty}$ and $(j_{i})=(k+1)^{\infty},$ we now show that there exists no $x\in S$ satisfying $r_{i}((0)^{\infty},x)=k+1$ for all $i$. Since $k+1$ is the earliest return times there exists a single interval $\mathcal{I}$ for which $\mathcal{I}:=\{x\in S: r_{1}((0)^{\infty},x)=k+1\},$ moreover for any $x\in \mathcal{I}$ we have $U_{\beta,0}(x)=(T_{1}^{k}\circ T_{0})(x)$. Thus, any $x$ satisfying $r_{i}((0)^{\infty},x)=k+1$ for all $i\in\mathbb{N}$ must satisfy
\begin{equation}
\label{Contradict}
(T_{1}^{k}\circ T_{0})^{i}(x)\in S \textrm{ for all }i\in\mathbb{N}.
\end{equation} We now explain why this is not possible.

The map $T_{1}^{k}\circ T_{0}$ scales distances by a factor $\beta^{k+1}$ and satisfies $(T_{1}^{k}\circ T_{0})(x)>x$ for $x$ to the right of the fixed point of $T_{1}^{k}\circ T_{0}.$ We previously observed that $(T_{1}^{k}\circ T_{0})(1/\beta)\in(\frac{1}{\beta},\frac{1}{\beta(\beta-1)}]$ thus the fixed point of $T_{1}^{k}\circ T_{0}$ is to the left of $S$. Therefore under repeated iteration of the map $T_{1}^{k}\circ T_{0}$ every $x\in S$ is eventually mapped outside of $S$. This implies that equation (\ref{Contradict}) cannot hold and we have proved our result.
\end{proof}
Combining Propositions \ref{Prop1}, \ref{Prop2} and \ref{Prop3} we conclude Theorem \ref{Free return times}.

\subsection{Proof of Theorem \ref{Restricted omega}}
We now prove Theorem \ref{Restricted omega} our proof is similar to Theorem \ref{Free return times} in that we make use of a nested interval construction. However, with our proof we do not explicitly construct the desired $\omega,$ we can only show existence, as such our proof takes on an added degree of abstraction.

Let us start by examining the consequences of $\beta\in(\alpha_{k},\eta_{k}]$ for some $k\geq 1.$ For $k\geq 2$ we ignore the intervals $(\alpha_{k},\gamma_{k}]$ as their proof is covered by Theorem \ref{Free return times}. For $\beta$ in the remaining parameter space the following inclusions hold
\begin{align}
\label{Cross over}
(T_{1}^{k}\circ T_{0})\Big(\frac{1}{\beta}\Big)&\in \Big(\frac{1}{\beta},\frac{1}{2(\beta-1)}\Big]\\
(T_{0}^{k}\circ T_{1})\Big(\frac{1}{\beta(\beta-1)}\Big)&\in\Big[\frac{1}{2(\beta-1)},\frac{1}{\beta(\beta-1)}\Big).\nonumber
\end{align} We emphasise that for any $\beta\in(1,2)$ the point $\frac{1}{2(\beta-1)}$ is the midpoint of the interval $S$ and is thus always in the interior of $S$. Equation (\ref{Cross over}) is equivalent to $U_{\beta,0}(\frac{1}{\beta})$ being contained in the left hand side of $S$, and $U_{\beta,1}(\frac{1}{\beta(\beta-1)})$ being contained in the right hand side of $S$. As such the two orbits cross over when they return to $S$.

The cross over property described by equation (\ref{Cross over}) implies
\begin{equation}
\label{Cross over consequence}
(T_{1}^{k}\circ T_{0})(B_{k+1}^0)\cup (T_{0}^{k}\circ T_{1})(B_{k+1}^1)=S.
\end{equation}Moreover, for any $i\geq k+1$ we have
\begin{equation}
\label{Full for large times}
(T_{1}^{i}\circ T_{0})(B_{i+1}^0)=S
\end{equation}
With the identities (\ref{Cross over consequence}) and (\ref{Full for large times}) we may now prove Theorem \ref{Restricted omega}.

\begin{proof}[Proof of Theorem \ref{Restricted omega}]
Let $\beta\in(\gamma_{k},\eta_{k}]$ and let us fix a sequence of return times $(j_{i})\in R_{\beta}^{\mathbb{N}}=\{k+1,k+2,\ldots\}^{\mathbb{N}}.$ We will construct a set $J,$ such that for any $x\in J$ there exists a sequence $\omega$ satisfying $r_{i}(\omega,x)=j_{i}$ for all $i\in\mathbb{N}.$ We construct $J$ by building a sequence of levels $J_{1},J_{2},\ldots$. Each $J_{i}$ will denote a finite collection of compact intervals $\{\mathcal{I}_{l}^{i}\}_{l=1}^{2^{i}}$.  Moreover,
\begin{equation}
\label{Inclusion}
\bigcup_{l=1}^{2^{i+1}}\mathcal{I}_{l}^{i+1} \subseteq \bigcup_{l=1}^{2^{i}}\mathcal{I}_{l}^{i}
 \end{equation} for each $i=1,2,\ldots$. Thus $$J=\bigcap_{i=1}^{\infty}\bigcup_{l=1}^{2^{i}}\mathcal{I}_{l}^{i}$$ is nonempty, and as we will see, for each $x\in J$ there exists an $\omega\in \Omega$ such that $r_{i}(\omega,x)=j_{i}$ for $i=1,2,\ldots.$ We emphasise that in our construction not every $\mathcal{I}_{j}^{i}$ will necessarily be nonempty.

For each level $J_{i}$ it is useful to define a collection of maps $M_{i}=\{f_{l}^{i}\}^{2^{i}}_{l=1}$. Each $f_{l}^{i}$ will be a map from $\mathcal{I}_{l}^{i}$ into $S$. These maps will also have the property that
\begin{equation}
\label{covering equation}
\bigcup_{l=1}^{2^{i}}f_{l}^{i}(\mathcal{I}_{l}^{i})=S.
\end{equation}
We start by letting
$$J_{1}=\{B_{j_{1}}^{0},B_{j_{1}}^{1}\} \textrm{ and } M_{1}=\{T_{1}^{j_{1}-1}\circ T_{0},T_{0}^{j_{1}-1}\circ T_{1}\}.$$ By Equations (\ref{Cross over consequence}) and (\ref{Full for large times}) we have $$(T_{1}^{j_{1}-1}\circ T_{0})(B_{j_{1}}^{0})\cup (T_{0}^{j_{1}-1}\circ T_{1})(B_{j_{1}}^{1})=S$$
So we satisfy (\ref{covering equation}) when $i=1$. Assume we have constructed $J_{i}$ and $M_{i}$ for $1\leq i\leq N,$ and (\ref{Inclusion}) holds for $1\leq i\leq N-1,$ and (\ref{covering equation}) holds for $1\leq i\leq N$. We now construct $J_{N+1}$ and $M_{N+1}.$ To each $f_{l}^{N}\in M_{N}$ we associate the compact intervals $(f_{l}^{N})^{-1}(B_{j_{N+1}}^{0})$ and $(f_{l}^{N})^{-1}(B_{j_{N+1}}^{1}),$ the set of these new intervals is our $J_{N+1}$. By (\ref{covering equation}) this  collection of intervals $\{(f_{l}^{N})^{-1}(B_{j_{N+1}}^{0}),(f_{l}^{N})^{-1}(B_{j_{N+1}}^{1})\}$ is nonempty. Each $f_{l}^{N}$ is a map from $\mathcal{I}_{l}^{N}$ into $S$, thus $(f_{l}^{N})^{-1}(B_{j_{N+1}}^{1})\subseteq \mathcal{I}_{l}^{N}$ and we have that equation (\ref{Inclusion}) holds for $i=N$.

To each $(f_{l}^{N})^{-1}(B_{j_{N+1}}^{0})$ we associate the map $(T_{1}^{j_{N+1}-1}\circ T_{0})\circ f_{l}^{N},$ and to each $(f_{l}^{N})^{-1}(B_{j_{N+1}}^{1})$ we associate the map $(T_{0}^{j_{N+1}-1}\circ T_{1})\circ f_{l}^{N}$ respectively. This collection of maps is our new $M_{N+1}.$

Moreover
\begin{align*}
&\Big(\bigcup_{l=1}^{2^{N}}((T_{1}^{j_{N+1}-1}\circ T_{0})\circ f_{l}^{N})\circ (f_{l}^{N})^{-1}(B_{j_{N+1}}^{0})\Big)\cup \Big(\bigcup_{l=1}^{2^{N}}((T_{0}^{j_{N+1}-1}\circ T_{1})\circ f_{l}^{N})\circ (f_{l}^{N})^{-1}(B_{j_{N+1}}^{1})\Big)\\
&=(T_{1}^{j_{N+1}-1}\circ T_{0})\Big(\bigcup_{l=1}^{2^{N}}f_{l}^{N}(f_{l}^{N})^{-1}(B_{j_{N+1}}^{0})\Big)\cup (T_{0}^{j_{N+1}-1}\circ T_{1})\Big(\bigcup_{l=1}^{2^{N}}f_{l}^{N}(f_{l}^{N})^{-1}(B_{j_{N+1}}^{1})\Big)\\
&=(T_{1}^{j_{N+1}-1}\circ T_{0})(B_{j_{N+1}}^{0})\cup (T_{0}^{j_{N+1}-1}\circ T_{1})(B_{j_{N+1}}^{1})\,\, (\textrm{ By }(\ref{covering equation}) \textrm{ for }i=N)\\
&=S\,\, (\textrm{ By }(\ref{Cross over consequence})  \textrm{ and } (\ref{Full for large times})).
\end{align*}
Therefore we satisfy (\ref{covering equation}) for $i=N+1$. As such we can repeat the above steps indefinitely and $J_{i}$ and $M_{i}$ are well defined for all $i\in\mathbb{N}$ and satisfy equations (\ref{Inclusion}) and (\ref{covering equation}). This implies that the set $J$ is well defined and nonempty.

It is not immediately obvious why an $x\in J$ admits an $\omega\in\Omega$ such that $r_{i}(\omega,x)=j_{i}$ for all $i\geq 1$. We now explain why. If $x\in J,$ then by our construction for each $n\in\mathbb{N}$ there exists $(\omega_{i}^{n})_{i=1}^{n}\in \{0,1\}^{n}$ such that
\begin{equation}
\label{Steps}
(T_{\overline{\omega_{i}^{n}}}^{j_{i}}\circ T_{\omega_{i}^{n}})\circ \cdots \circ  (T_{\overline{\omega_{1}^{n}}}^{j_{1}}\circ T_{\omega_{1}^{n}})(x)\in S
 \end{equation}for all $1\leq i\leq n$. We identify the finite sequence $(\omega_{i}^{n})$ with the infinite sequence $\upsilon_{n}=(\omega_{1}^{n},\ldots,\omega_{n}^{n},(0)^{\infty}).$ We equip $\Omega$ with the usual metric $d(\cdot,\cdot)$ where $d((\epsilon_{i}),(\delta_{i}))=2^{-n((\epsilon_{i}),(\delta_{i}))}$ where $n(x,y)=\inf\{i:\epsilon_{i}\neq \delta_{i}.$ With respect to this metric $\Omega$ is a compact metric space, thus there exists $\upsilon\in\Omega$ and a subsequence of the $(\upsilon_{n})$ such that $\upsilon_{n_{k}}\to \upsilon.$ This $\upsilon$ has the property that
 \begin{equation}
 \label{upsilon steps}
 (T_{\overline{\upsilon_{i}}}^{j_{i}}\circ T_{\upsilon_{i}})\circ \cdots \circ  (T_{\overline{\upsilon_{1}}}^{j_{1}}\circ T_{\upsilon_{1}})(x)\in S
  \end{equation}for all $i\in \mathbb{N}$. (\ref{upsilon steps}) is a consequence of $\upsilon$ being the limit of sequences satisfying $(\ref{Steps}).$ Clearly (\ref{upsilon steps}) implies that $r_{i}(\upsilon,x)=j_{i}$ for all $i\in\mathbb{N}$.
\end{proof}
\begin{Remark}
We end this section by pointing out that there are non trivial examples of $\beta\in(1,2)$ for which there exists $(j_{i})\in R_{\beta}^{\mathbb{N}}$ and no $x\in S$ and $\omega\in\Omega$ for which $r_{i}(\omega,x)=j_{i}$ for all $i\in\mathbb{N}$. For example take $\beta=1.754$. We chose $\beta$ to be this value because it is slightly less than $\alpha_{2}.$ Thus $T_{1}\circ T_{0}(\frac{1}{\beta})\in S,$ but it is only slightly less than the right end point of the switch region. Clearly $R_{\beta}:=\{2,3,\ldots\}.$ However, any point that can have a return time two gets mapped close to the endpoints of $S$ under the corresponding map. Being close to the endpoints of the switch suggests either a large return time or a small return time. This is the case for $\beta=1.754$, and a simple calculation shows that it is not possible for $r_{1}(\omega,x)=2$ and $r_{2}(\omega,x)=3.$
\end{Remark}
\section{Proof of Theorem \ref{Luroth return map}}

Let us begin our proof of Theorem \ref{Luroth return map} by defining the set $M$ that appear in its statement. Let $$M:=\Big\{\beta\in(1,2): \textrm{card } \Sigma_{\beta}(1)=1\Big\}\cup \Big\{\beta\in(1,2): U_{\beta,0}\Big(\frac{1}{\beta}\Big)\in\Big\{\frac{1}{\beta},\frac{1}{\beta(\beta-1)}\Big\}\Big\}.$$ The first set in this union is the set of univoque bases, the study of this set is classical within expansions in noninteger bases, we refer the reader to the following papers for more on this subject \cite{dVK,EHJ,EJ,KL}. In \cite{EJ} Erd\H os and Jo\'o showed that the set of univoque bases has Hausdorff dimension $1$ and Lebesgue measure zero. The second set in the above union is a countable set of algebraic numbers, thus $M$ has Hausdorff dimension $1$ and Lebesgue measure zero. It is worth noting that if $\beta\in \{\beta\in(1,2): U_{\beta,0}(\frac{1}{\beta})\in\{\frac{1}{\beta},\frac{1}{\beta(\beta-1)}\}\}$ then $\textrm{card } \Sigma_{\beta}(1)=\aleph_{0}$. The important observation to make from the definition of $M$ is that the following statement holds
$$\beta\in M\iff \frac{1}{\beta} \textrm{ and }\frac{1}{\beta(\beta-1)}\textrm{ are never mapped into the interior of }S.$$ This property will be sufficient to prove that both $U_{\beta,0}$ and $U_{\beta,1}$ are GLSTs. Our proof of Theorem \ref{Luroth return map} is split over the following propositions.

\begin{Proposition}
\label{Not Luroth prop}
If $\beta\notin M$ then $U_{\beta,0}$ and $U_{\beta,1}$ are not GLSTs.
\end{Proposition}
\begin{proof}
If $\beta\notin M$ then $U_{\beta,0}(\frac{1}{\beta})\in S^{0}.$ In which case at the left endpoint of $S$ the graph of $U_{\beta,0}$ has an incomplete branch. Thus it is not possible that $U_{\beta,0}$ is a GLST as all of the branches are full for this class of transformation. The proof that $U_{\beta,1}$ is not a GLST is similar and appeals to the fact that $U_{\beta,1}(\frac{1}{\beta(\beta-1)})\in S^{0}.$
\end{proof}

\begin{Proposition}
\label{Luroth prop}
If $\beta\in M$ then $U_{\beta,0}$ and $U_{\beta,1}$ are GLSTs.
\end{Proposition}
We will only show that if $\beta\in M$ then $U_{\beta,0}$ is a GLST, the proof for $U_{\beta,1}$ being analogous. Moreover, as we previously demonstrated in Example \ref{Golden ratio example} that the maps $U_{\beta,0}$ and $U_{\beta,1}$ were GLSTs for $\beta=\frac{1+\sqrt{5}}{2}$ we restrict our attention to the interval $(\frac{1+\sqrt{5}}{2},2),$ where the rest of the set $M$ exists.

Before proceeding with our proof that $U_{\beta,0}$ is a GLST we make several observations. Let $\beta\in(\frac{1+\sqrt{5}}{2},2)$ and $x\in S$ be such that $U_{\beta,0}$ is well defined, then
\begin{equation}
\label{Hops equation}
U_{\beta,0}(x)=(T_{\omega_{i}}^{n_{i}}\circ \cdots  \circ T_{1}^{n_{1}}\circ T_{0})(x)
 \end{equation}for some $\omega_{i}\in\{0,1\}$ that alternate digits with $\omega_{1}=1$. Equation (\ref{Hops equation}) holds because the map $T_{0}$ maps every element of $S$ outside of $S$. The quantity $i-1$ is the number of times $x$ jumps over $S$ before eventually being mapped inside. Note that if $i$ is even then $\omega_{i}=0$ and if $i$ is odd then $\omega_{i}=1.$

Let $$C_{n}:=T^{-n}_{0}(S) \textrm{ and } D_{n}:=T^{-n}_{1}(S)$$ where $n\in\mathbb{N}$. Equation (\ref{Hops equation}) demonstrates that if $U_{\beta,0}(x)$ is well defined then $x$ must eventually map into a $C_{n}$ or a $D_{n}$. Note that for $\beta\in(\frac{1+\sqrt{5}}{2},2)$ the $C_{n}$ are all disjoint and contained in the interval $(0,\frac{1}{\beta})$, and similarly the $D_{n}$ are all disjoint and contained in $(\frac{1}{\beta(\beta-1)},\frac{1}{\beta-1})$

It is instructive here to make a final notational remark before we give our proof. As we will see, the proof of Proposition \ref{Luroth prop} relies heavily on understanding the trajectories of certain intervals under certain maps and where they lie relative to $C_{n},D_{n}$ and $S$. Often we will be in a situation where a relation ($I\cap J=\emptyset$, $I\subseteq J$) is true only if we ignore the endpoints of these intervals. For ease of exposition instead of repeatedly emphasising the fact that this relation holds modulo the endpoints we will simply state that the equation holds. This is technically not correct, but our proof still holds and is far more succinct by adopting this convention.

\begin{proof}[Proof of Proposition \ref{Luroth prop}]
To prove $U_{\beta,0}$ is a GLST it suffices to show that for any $x\in S$ such that $U_{\beta,0}(x)$ is well defined then we have
\begin{equation}
\label{Sufficient equation}
\{y\in S: U_{\beta,0}(y)=(T_{\omega_{i}}^{n_{i}}\circ \cdots  \circ T_{1}^{n_{1}}\circ T_{0})(y)\}=(T_{\omega_{i}}^{n_{i}}\circ \cdots  \circ T_{1}^{n_{1}}\circ T_{0})^{-1}(S).
 \end{equation}
 Where we have assumed $U_{\beta,0}(x)=(T_{\omega_{i}}^{n_{i}}\circ \cdots  \circ T_{1}^{n_{1}}\circ T_{0})(x)$. We now explain why Equation (\ref{Sufficient equation}) implies $U_{\beta,0}$ is a GLST. The intervals on the left hand side of equation (\ref{Sufficient equation}) are all disjoint, thus we satisfy part $(1)$ of the definition of a GLST. By Sidorov's result we know that for Lebesgue almost every $x\in S$ the map $U_{\beta,0}(x)$ is well defined, thus the lengths of the intervals on the left hand side of equation (\ref{Sufficient equation}) sum up to equal the length of $S$ and we satisfy part $(2)$ of the definition of a GLST. Lastly, the right hand side of equation (\ref{Sufficient equation}) demonstrates that $U_{\beta,0}$ restricted to this interval is surjective onto $S,$ since there is a unique surjective linear orientation preserving map from this interval onto $S$ we also satisfy part $(3)$ of the definition of a GLST.

We begin with the most simple case, we assume that $U_{\beta,0}(x)=(T_{1}^{n_{1}}\circ T_{0})(x)$, i.e. $T_{0}(x)\in D_{n_{1}}$. Importantly, since $\beta\in M$ we know that $1 \notin D_{n_{1}}^{0}$.
Thus $T_{0}(S)\cap D_{n_{1}}=[1,\frac{1}{\beta-1}]\cap D_{n_{1}}=D_{n_{1}}.$ Therefore $T_{0}^{-1} (D_{n_{1}})\subseteq S$ and any $y$ in this interval satisfies $U_{\beta,0}(y)= (T_{1}^{n_{1}}\circ T_{0})(y)$.
This implies that
\begin{equation}
\label{Full branch 0}
\{y\in S:U_{\beta,0}(x)=(T^{n_{1}}_{1}\circ T_{0})(y)\}= (T^{n_{1}}_{1}\circ T_{0})^{-1}(S).
\end{equation}

It remains to show that equation (\ref{Sufficient equation}) holds in the general case. Obviously
\begin{equation}
\label{Obvious inclusion}
\{y\in S: U_{\beta,0}(y)=(T_{\omega_{i}}^{n_{i}}\circ \cdots  \circ T_{1}^{n_{1}}\circ T_{0})(y)\}\subseteq(T_{\omega_{i}}^{n_{i}}\circ \cdots  \circ T_{1}^{n_{1}}\circ T_{0})^{-1}(S).
\end{equation}So we have to show that the opposite inclusion holds, for this we examine the formula for $U_{\beta,0}$ more closely. We assume $U_{\beta,0}(x)=(T_{\omega_{i}}^{n_{i}}\circ \cdots  \circ T_{1}^{n_{1}}\circ T_{0})(x)$ for some $i\geq 2$.
Since $i\geq 2$ we have $T_{0}(x)$ is contained in a connected component of $[1,\frac{1}{\beta-1})\setminus \cup_{n=1}^{\infty}D_{n}.$ Let us denote this interval by $\mathcal{I}_{1}.$ We also let $$E:=\Big\{T_{0}^{-n}\Big(\frac{1}{\beta}\Big),  T_{0}^{-n}\Big(\frac{1}{\beta(\beta-1)}\Big),T_{1}^{-n}\Big(\frac{1}{\beta}\Big),  T_{1}^{-n}\Big(\frac{1}{\beta(\beta-1)}\Big),G_{\beta}^{n}(1),G_{\beta}^{n}\Big(\frac{1}{\beta-1}-1\Big):n\geq 0 \Big\}.$$ Here $G_{\beta}$ is the greedy map defined earlier. Since $\beta\in M$ no element of $E$ is contained in the interior of a $C_{n},$ a $D_{n}$, or $S$.

Importantly $\mathcal{I}_{1}=(a_{1},b_{1})$ where $a_{1},b_{1}\in E.$ In this case either $$(a_{1},b_{1})=\Big(1,T_{1}^{-n_{1}}\Big(\frac{1}{\beta}\Big)\Big)\textrm{ or } (a_{1},b_{1})=\Big(T_{1}^{-(n_{1}-1)}\Big(\frac{1}{\beta(\beta-1)}\Big),T_{1}^{-n_{1}}\Big(\frac{1}{\beta}\Big)\Big).$$ Therefore
$$T_{1}^{k}(\mathcal{I}_{1})\cap S=\emptyset \textrm{ for }1\leq k\leq n_{1}-1 \textrm{ and }T_{1}^{n_{1}}(\mathcal{I}_{1})\subseteq \Big(\frac{2-\beta}{\beta-1},\frac{1}{\beta}\Big).$$
The endpoints of $T_{1}^{n_{1}}(\mathcal{I}_{1})$ are elements of $E$ and are therefore not contained in the interior of any $C_{n}$. Either $(T^{n_1}_{1}\circ T_{0})(x)\in C_{n}$ for some $n$ or maybe $(T^{n_1}_{1}\circ T_{0})(x)\in T_{1}^{n_{1}}(\mathcal{I}_{1})\setminus \cup_{n=1}^{\infty} C_{n}$. If $(T^{n_1}_{1}\circ T_{0})(x)\in  T_{1}^{n_{1}}(\mathcal{I}_{1})\setminus \cup_{n=1}^{\infty} C_{n}$ then let the connected component it is contained in be denoted by $\mathcal{I}_{2}.$ Let $\mathcal{I}_{2}=(a_{2},b_{2})$ then again $a_{2},b_{2}\in E.$ In which case
\begin{equation}
T_{0}^{k}(\mathcal{I}_{2})\cap S=\emptyset \textrm{ for }1\leq k\leq n_{2}-1 \textrm{ and } T_{0}^{n_{2}}(\mathcal{I}_{2})\subseteq \Big(\frac{1}{\beta(\beta-1)},1\Big)
\end{equation}The endpoints of $T_{0}^{n_{2}}(\mathcal{I}_{2})$ are again contained in $E$ and therefore do not intersect the interior of any $D_{n}.$ The point $x$ has either been mapped into a $D_{n}$ or is contained in a connected component of $T_{0}^{n_{2}}(\mathcal{I}_{2})\setminus \cup_{n=1}^{\infty}D_{n}$. If it is contained in a connected component of $T_{0}^{n_{2}}(\mathcal{I}_{2})\setminus \cup_{n=1}^{\infty}D_{n}$ then we repeat the previous steps. Eventually $x$ is mapped into either $C_{n_i}$ or $D_{n_i}$ and our algorithm terminates. Without loss of generality we assume $x$ is eventually mapped into $D_{n_i}$. The above algorithm yields a finite sequence of intervals $(\mathcal{I}_{j})_{j=1}^{i-1}$ which satisfy the following properties:
\begin{enumerate}
  \item $\mathcal{I}_{1}\subseteq T_{0}(S).$
  \item For $1\leq j\leq i-1$ $$T_{\omega_{j}}^{k}(\mathcal{I}_{n_{j}})\cap S=\emptyset \textrm{ for }1\leq k\leq n_{j} $$
  \item For $1\leq j\leq i-2$ we have $\mathcal{I}_{j+1}\subseteq T_{\omega_{j}}^{n_{j}}(\mathcal{I}_{j})$
  \item $$D_{n_i}\subseteq T_{\omega_{i-1}}^{n_{i-1}}(\mathcal{I}_{i-1}).$$
\end{enumerate}
Where in the above $\omega_{j}=0$ if $j$ is even and $\omega_{j}=1$ if $j$ is odd. These properties have the following consequences:
\begin{enumerate}
\setcounter{enumi}{4}
 \item $(T_{\omega_{i}}^{n_{i}})^{-1}(S)\subseteq \mathcal{I}_{i-1}$
 \item For $1\leq j\leq i-1$ $$(T_{\omega_{i}}^{n_{i}}\circ \cdots \circ T_{\omega_{j}}^{k})^{-1}(S)\cap S=\emptyset\textrm{ for }1\leq k\leq n_{j}.$$
 \item For $1\leq j\leq i-1$ $$(T_{\omega_{i}}^{n_{i}}\circ \cdots \circ T_{\omega_{j}}^{n_{j}})^{-1}(S)\subseteq \mathcal{I}_{n_{j}}$$
\item $$(T_{\omega_{i}}^{n_{i}}\circ \cdots \circ T_{1}^{n_{1}}\circ T_{0})^{-1}(S)\subseteq S.$$
\end{enumerate}
Property $(8)$ states that $(T_{\omega_{i}}^{n_{i}}\circ \cdots  \circ T_{1}^{n_{1}}\circ T_{0})^{-1}(S)\subseteq S$. Moreover, properties $(5)$, $(6)$ and $(7)$ imply that every $y\in(T_{\omega_{i}}^{n_{i}}\circ \cdots  \circ T_{1}^{n_{1}}\circ T_{0})^{-1}(S)$ satisfies $U_{\beta,0}(y)=(T_{\omega_{i}}^{n_{i}}\circ \cdots \circ T_{1}^{n_{1}}\circ T_{0})(y).$ Thus
$$(T_{\omega_{i}}^{n_{i}}\circ \cdots  \circ T_{1}^{n_{1}}\circ T_{0})^{-1}(S)\subseteq\{y\in S: U_{\beta,0}(y)=(T_{\omega_{i}}^{n_{i}}\circ \cdots  \circ T_{1}^{n_{1}}\circ T_{0})(y)\},$$ which when combined with equation (\ref{Obvious inclusion}) yields (\ref{Sufficient equation}).

\end{proof}

\end{document}